\documentclass{article}
\pdfoutput=1
\usepackage{graphicx,tikz,graphics,float,fp,amsthm,cite,multicol,subcaption,caption,amssymb,mathtools,hyphenat,multirow,longtable,array}
\usetikzlibrary{automata,positioning, decorations.markings,fixedpointarithmetic,arrows}
\usepackage{longtable,comment}
\usepackage[shortlabels]{enumitem}
\usepackage{url}
\usetikzlibrary{external}
\newtheorem{theorem}{Theorem}[section]

\newtheorem{question}[theorem]{Question}
\newtheorem{proposition}[theorem]{Proposition}

\theoremstyle{definition}
\newtheorem*{example*}{Example}

\usetikzlibrary{decorations.pathreplacing,decorations.markings,arrows,calc,math,arrows.meta}
\newcolumntype{C}{>{$}c<{$}}
\newcolumntype{M}{>{$}p{10cm}<{$}}

\begin{document}

\author{Danny Dyer\\
	Department of Mathematics and Statistics\\
	St. John's Campus, Memorial University of Newfoundland\\
	St. John's, Newfoundland\\
	Canada\\
	{\tt dyer@mun.ca}
	\and
	Jared Howell\\
	School of Science and the Environment\\
	Grenfell Campus, Memorial University of Newfoundland\\
	Corner Brook, Newfoundland\\
	Canada\\
	{\tt jahowell@grenfell.mun.ca}
	\and
	Brittany Pittman\\
	Department of Mathematics and Statistics\\
	St. John's Campus, Memorial University of Newfoundland\\
	St. John's, Newfoundland\\
	Canada\\
	{\tt bep275@mun.ca}}

\title{A note on watchman's walks in de Bruijn graphs}
\date{ }

\maketitle

\begin{abstract}
The watchman's walk problem in a digraph calls for finding a minimum length closed dominating walk, where direction of arcs is respected. The watchman's walk of a de Bruijn graph of order $k$ is described by a de Bruijn sequence of order $k-1$. This idea is extended to certain subdigraphs of de Bruijn graphs.
\end{abstract}

\section{Introduction}

In 1946, de Bruijn studied the problem of finding a shortest possible binary sequence that contains every binary string of length $k$. In \cite{debruijn1}, he solved this problem and introduced both de Bruijn sequences and de Bruijn graphs. A \textit{de Bruijn sequence of order $k$} is a binary sequence of length $2^k$ such that the last bit is said to be adjacent to the first bit and every binary $k$-tuple occurs exactly once in the sequence. For example, a de Bruijn sequence of order $2$ is 1001. Every binary $2$-tuple ($00$, $01$, $10$, and $11$) occurs in this sequence.

In a de Bruijn sequence $S$ of order $k$, the \textit{k-tour} is the sequence of substrings of length $k$, in order of their occurrence in $S$, starting with the initial $k$-string in $S$. Since $S$ is considered to be cyclic, there are $2^k$ substrings in the tour.

A de Bruijn graph of order $k$ is a directed graph, denoted $G(k)$, whose $2^k$ vertices are labelled by each possible binary string of length $k$. To define the edges one must consider two possible left shift operations. A left shift operation occurs when a $k$-bit substring $b_1 b_2 \ldots b_k$ is obtained from another substring $a_1 a_2 \ldots a_k$ such that $b_i=a_{i+1}$ for $1\leq i<k$. A \textit{cycle shift} occurs when $a_1 a_2 \ldots a_k \rightarrow b_1 b_2 \ldots b_k$, and $b_k=a_1$. For example, $100 \rightarrow 001$. A \textit{de Bruijn shift} occurs when $a_1 a_2 \ldots a_k \rightarrow b_1 b_2 \ldots b_k$, and $b_k\not=a_1$. For example, $100 \rightarrow 000$. In the de Bruijn graph of order $k$ there is an arc from the vertex labelled by string $a$ to the vertex labelled by string $b$ if and only if $b$ can be obtained from $a$ using one of the left shift operations. The de Bruijn graphs of order 2 and 3 are shown in Figures \ref{debruijngraph2} and \ref{debruijngraph3}.

\begin{figure}[H]
	\centering
	\begin{subfigure}[b]{0.4\textwidth}
		\centering
		\includegraphics[scale=1]{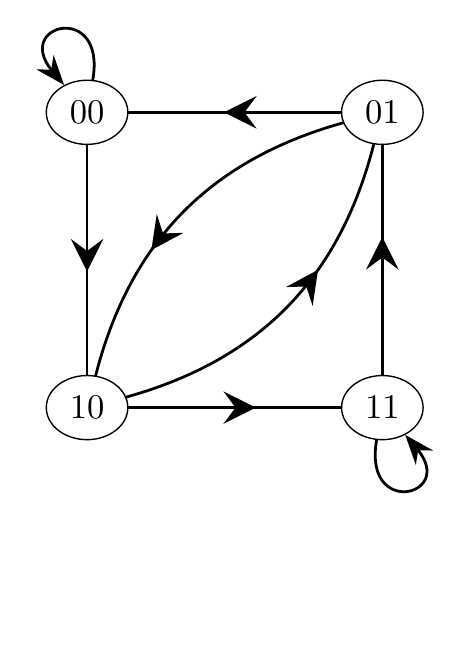}
		\caption{A de Bruijn graph of order $2$ }
		\label{debruijngraph2}
	\end{subfigure}\hspace{6ex}
	\begin{subfigure}[b]{0.4\textwidth}
		\centering
		\includegraphics[scale=0.9]{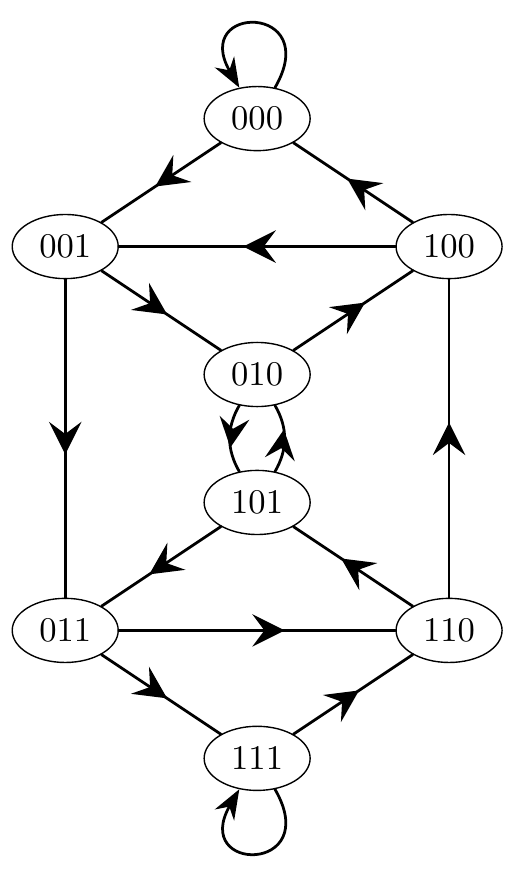}
		\caption{A de Bruijn graph of order $3$}
		\label{debruijngraph3}
	\end{subfigure}
\end{figure}

While the original de Bruijn sequences have only two symbols in their alphabet, the idea of a de Bruijn sequence, and hence de Bruijn graph, can be generalized to an alphabet of size $a\in \mathbb{N}$.  We denote the corresponding de Bruijn graph by $G(a,k)$. For example, the digraph in Figure~\ref{db2} is the de Bruijn graph of order $2$ on the alphabet $\{0,1,2\}$. A corresponding de Bruijn sequence is 220011210.
For further results on de Bruijn sequences and graphs on larger alphabets, see \cite{deb2, genome1, register, multidb}.

\begin{figure}[H]

		\centering
		\includegraphics[scale=0.9]{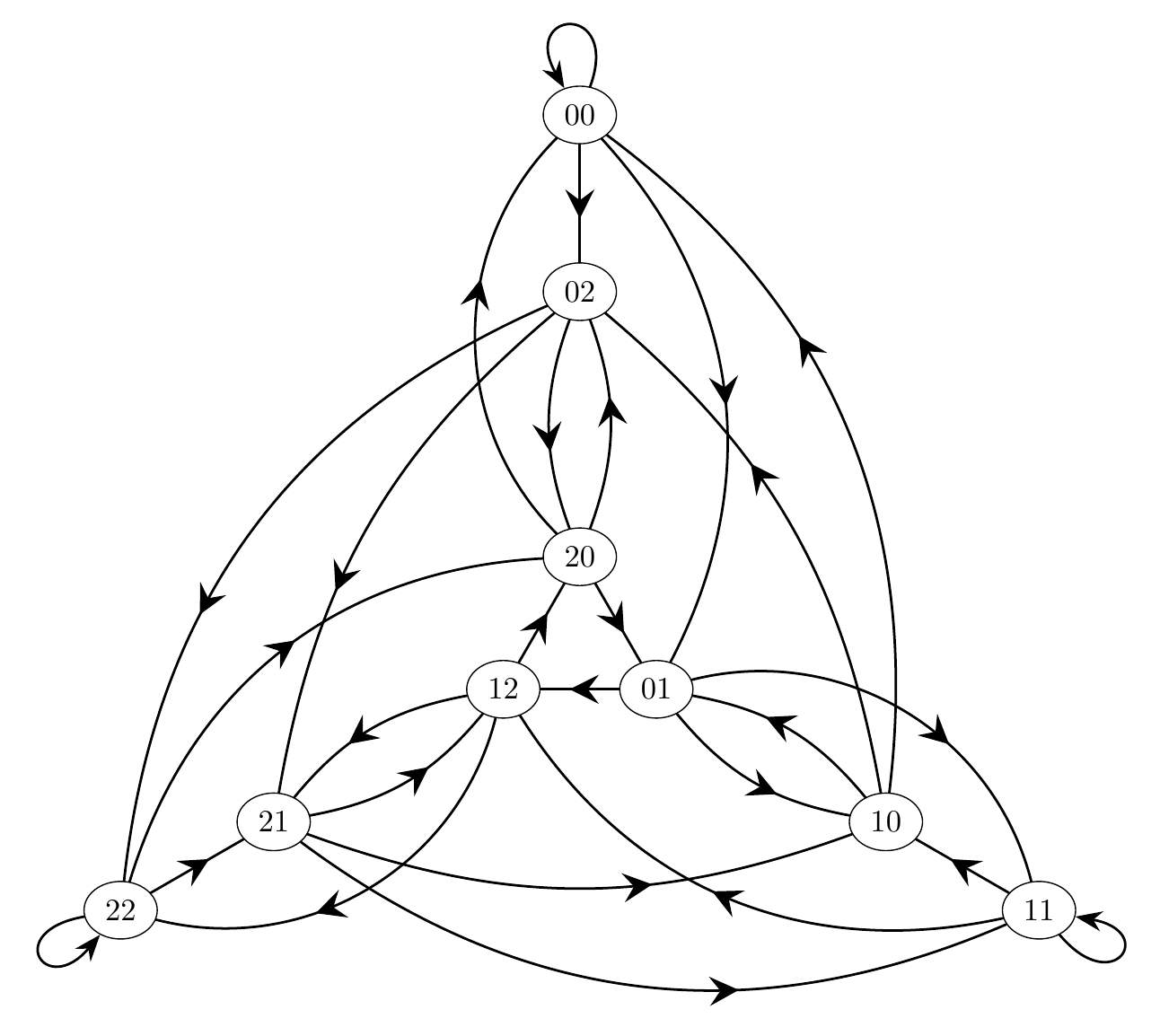}
		\caption{The de Bruijn graph $G(3,2)$ on $\{0,1,2\}$}
		\label{db2}

\end{figure}

Given a graph $G$ and a single watchman, the watchman's walk problem looks at the scenario of that watchman traversing the graph in such a way that the route is a minimum closed dominating walk. Any such walk in $G$ is called a \textit{watchman's walk} of $G$. The length of a watchman's walk in $G$ is denoted by $w(G)$; this is called the \textit{watchman number} of $G$. The typical goals of the watchman's walk problem are to find a watchman's walk of $G$ and determine $w(G)$. The watchman's walk problem was introduced by Hartnell, Rall, and Whitehead in 1998 in \cite{intro}. In a directed graph, we say that the watchman can only move to and see the vertices that are adjacent to him relative to outgoing arcs. That is, a watchman's walk is oriented and domination occurs in the direction of the arcs. More results on directed graphs may be found in \cite{britthesis, tournamentpaper}. We follow \cite{west}, with the exception of our use of $\deg(v)$ for the degree of a vertex $v$; basic definitions and notation can be found there.

The watchman's walk problem is a variation of the domination problem. Domination in graphs was formally introduced in 1958 by Berge in \cite{berge}. The domination problem in graphs aims to find a dominating set; a set of vertices in a graph such that every vertex of the graph is either in that set or a neighbour of a vertex in that set. Directed domination in tournaments was first considered by Erd{\"o}s in \cite{erdos}. A survey of results on the domination number of graphs can be found in \cite{domfundamentals}.

\section{Watchman's walks in de Bruijn graphs}
 In Figure \ref{debruijngraph3} we can see that $G(3)$ has a watchman's walk, namely the cycle given by $100,001,011,110,100$. In the following theorem, we show that this is also true of $G(k)$ for any value of $k$ on any alphabet. In Theorem \ref{debruijnwwlength}, we give an expression for the watchman number of $G(k)$ for any $k$ on any alphabet.

\begin{theorem} \label{debruijnwwlength}
If $G$ is the de Bruijn graph of a de Bruijn sequence of order $k$ on an alphabet of size $a$, then $w(G)=a^{k-1}$.
\end{theorem}

\begin{proof}
Consider $G(a, k-1)$. This graph has $a^{k-1}$ vertices. The $(k-1)$-tour of the corresponding de Bruijn sequence $S$ corresponds to a Hamilton cycle of the digraph, as each substring is included exactly once in the tour, and there is an arc from the vertex corresponding to any substring to the vertex corresponding to the following substring in the tour.
We will show that the Hamilton cycle $H$ generated by $S$ corresponds to a watchman's walk of $G(a,k).$

Consider $G=G(a, k)$. If we again consider the $(k-1)$-tour of the sequence $S$ of order $k-1$, we can generate a set $D$ of $a^{k-1}$ strings of length $k$. For each string $v \in D$, we obtain a $(k-1)$-string $v'$ by removing the first bit of a string $v \in D$. As $D$ was a de Bruijn sequence of order $k-1$, these $(k-1)$-strings are all distinct. By appending an element of the alphabet $A$ to the end of $v'$ for each $v\in D$, we obtain the $a$ neighbours of $v$ in $G$. Since the $(k-1)$ strings were distinct, the neighbourhoods in $G$ of any two elements of $D$ are disjoint. Thus, we get that the union of the neighbourhoods of the elements of $D$ is all of the $a^k$ vertices in $G$. Thus, any vertex in $G$ is either in $D$ or is dominated by an element of $D$. This means that $D$ is a dominating set in $G$.

If we again consider the $k$-tour of $S$, each $k$-substring in this tour can be obtained by a cycle shift or de Bruijn shift from the previous string. As the $k$-tour is considered cyclically, this tour corresponds to a closed dominating walk in $G$. Since $D$ is also a dominating set of $G$, $D$ induces a closed dominating walk in $G$.

Finally, $D$ must induce a minimum dominating cycle. Let $v$ be a vertex in $G$. Consider forming a dominating cycle $W$ that starts at $v$. Each vertex has at most $a$ out-neighbours, and one of these out-neighbours must also be in $W$. So, each vertex dominates at most $a-1$ vertices outside of $W$. So, if $W$ is dominating, $W$ must have length at least $\frac{a^k}{a}=a^{k-1}$. Now, $D$ induces a dominating cycle of length $a^{k-1}$, so $D$ induces a minimum dominating cycle. Since the length of $D$ is $a^{k-1}$ and $D$ is a watchman's walk for $G$, $w(G)=a^{k-1}$.
\end{proof}

\begin{example*}
	The de Bruijn sequence of order 2, 1001, has $3$-tour $\{100,001,$ $011,110\}$. This induces a watchman's walk in $G(3)$. See Figures \ref{debruijngraph2} and \ref{debruijngraph3}.
	
\end{example*}

\section{Watchman's walks in subdigraphs of de Bruijn graphs}

If $S$ is a de Bruijn sequence of order $k$ on an alphabet $A$, we call a sequence $D$ on the same alphabet $A$ a \textit{generating sequence} if $D$ has length at least $k$. We define the \textit{de Bruijn subdigraph generated by $D$} to be the digraph induced by the set of $k$-sequences in $D$ as well as the $k$-sequences obtainable from the $k$-sequences in $D$ via a left shift operation. We denote this by $G_D(|A|,k)$. If $D$ contains every possible $k$-string, then $D$ generates $G(|A|,k)$, otherwise the subdigraph generated by $D$ may be a proper subdigraph of the de Bruijn graph corresponding to $S$. In the latter case, we would like to know when the $k$-tour of $D$ induces a watchman's walk in $G_D(|A|,k)$. First we give two families of sequences that never induce a watchman's walk in the subdigraphs of the de Bruijn graph.

\begin{proposition}

If $D$ is a generating sequence of order $k$ on any alphabet $A$ of size at least $2$, then the walk induced by $D$ is never a watchman's walk in $G_D(|A|,k)$ if either of the following conditions hold:
\begin{enumerate}
	\item there are $k$ consecutive occurrences of the same bit in $D$; or
	\item $D$ is the concatenation of two identical sequences of length at least $k$.
\end{enumerate}

\end{proposition}

\begin{proof}
Let $D=a_1,a_2,\ldots, a_n$ and let the subdigraph generated by $D$ be $G'$.

For 1,  we take $a_{i+1}=a_{i+2}=\ldots=a_{i+k}$ and let $v_1=a_i,a_{i+1},\ldots, a_{i+k-1}$ and $v_2=a_{i+1},\ldots, a_{i+k}$. Since $a_{i+1}=a_{i+2}=\ldots=a_{i+k}$, we know that $(v_1,v_2)$ is an arc in $G'$, and $N\lbrack v_2\rbrack \subseteq N(v_1)$. Hence, we would not need to visit $v_2$ in a watchman's walk for $G'$. Thus, $D$ does not induce a watchman's walk for the subdigraph generated by $D$.

For 2,  we would only need to traverse at most half of the $k$-tour of the generating sequence to get a watchman's walk. Thus, the entire $k$-tour would be superfluous.
\end{proof}

In some cases, a sequence will always induce a watchman's walk. A family of such sequences is given in the following theorem.

\begin{theorem} \label{thm:generating}
If $D$ is a generating sequence of order $k$ on any alphabet $A$ of size at least $2$, such that there are no repeated $(k-1)$-tuples in $D$, then the walk induced by $D$ is a watchman's walk in $G_D(|A|,k)$.
\end{theorem}

\begin{proof}

In the subdigraph generated by $D$, the neighbourhood of each vertex is determined by the final $(k-1)$ bits in the $k$-substring that labels the vertex. If every $(k-1)$-tuple in $D$ is unique then each vertex in the subdigraph generated by $D$ will have only one out-neighbour in that subdigraph, and their out-neighbourhoods in $G_D(|A|,k)$ will be unique. Thus, if $D$ is a sequence of length $n$, then $G_D(|A|,k)$ will have $k \times n$ vertices. The vertices of the subdigraph generated by $D$ induce a closed walk in $G_D(|A|,k)$. This walk is clearly dominating, and has length $n$. In $G_D(|A|,k)$, each vertex has at most $k$ out-neighbours, so any closed dominating walk must have length at least $\frac{k \cdot n}{k}=n.$ Hence, the closed dominating walk induced by the vertices of the subdigraph generated by $D$ is a watchman's walk in $G_D(|A|,k)$.
\end{proof}

Theorem~\ref{thm:generating} is not an ``if and only if'', since sequences with repeated $(k-1)$-tuples may still lead to watchman's walks in the generated subdigraph. One such example is given in Figure~\ref{db3}. In this figure, the arcs between vertices of the generating sequence are in black, while the arcs involving outneighbours are in grey. There are exactly two watchman's walks in this graph of length 8, one of which is defined by the generating sequence.

\begin{figure}[H]
		\centering
		\includegraphics[scale=0.9]{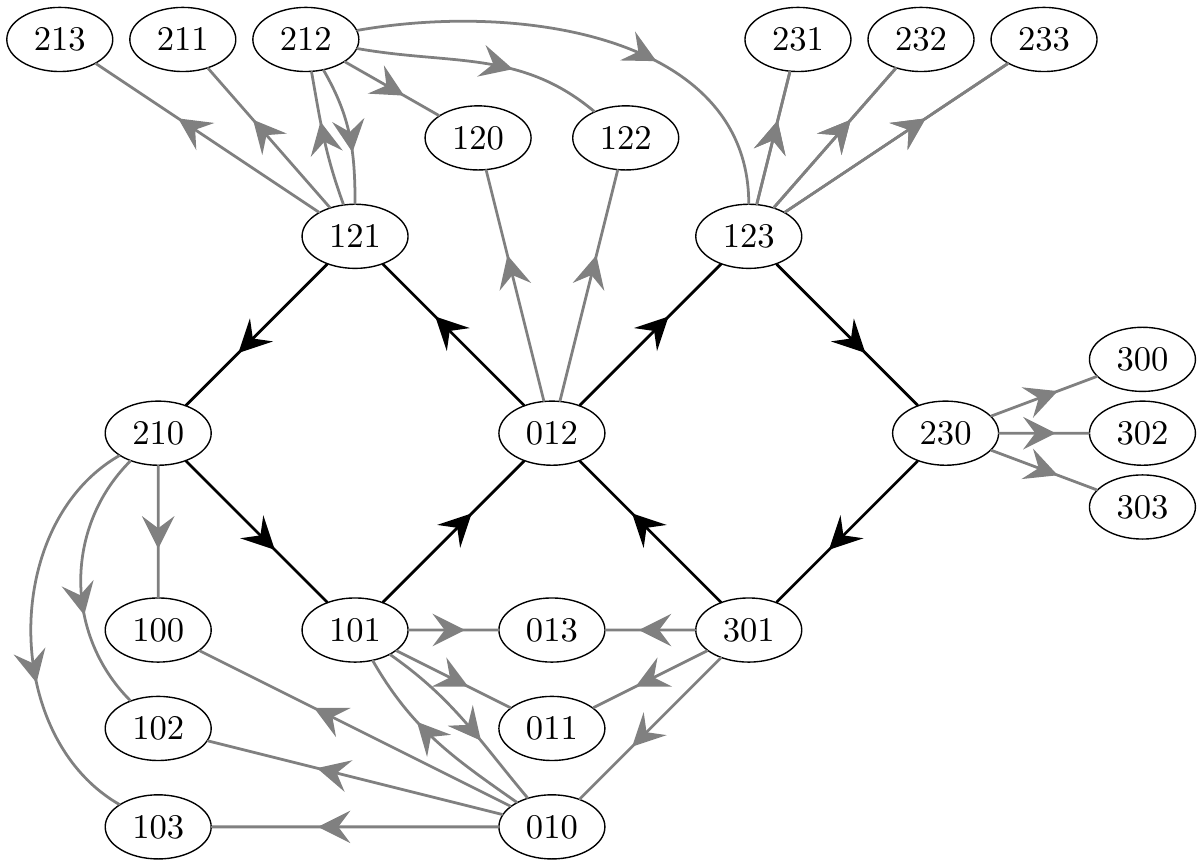}
		\caption{The subgraph of the de Bruijn graph $G(4,3)$ on $\{0,1,2,3\}$ generated by 01210123.}
		\label{db3}
\end{figure}

The open question remains.

\begin{question} What conditions must be placed on a generating sequence $D$, so that the walk induced by $D$ is a watchman's walk in the de Bruijn subdigraph generated by $D$?
\end{question}

\bibliographystyle{plain}
\bibliography{db}
\end{document}